\documentclass[11pt,envcountsame,envcountreset,envcountsect]{article}
\usepackage[english]{babel}
\usepackage[latin1]{inputenc}
\usepackage{graphicx}
\usepackage{amsmath,amssymb,amsthm,amsfonts}
\usepackage{amssymb}
\usepackage{amscd}
\usepackage{verbatim}
\usepackage{hyperref}
\hypersetup{colorlinks=false}
\newcommand{\ud}{\mathrm{d}}

\newtheorem{thm}{Theorem}[section]

\newtheorem{lemma}[thm]{Lemma}
\newtheorem{prop}[thm]{Proposition}

\theoremstyle{definition}

\newtheorem{rem}[thm]{Remark}

\numberwithin{equation}{section}

\DeclareMathSymbol{\C}{\mathalpha}{AMSb}{"43}

\newcommand{\R}{{\mathbb{R}}}

\newcommand{\supp}{\mathrm{supp}}
\newcommand{\bsub}{\begin{subequations}}
\newcommand{\esub}{\end{subequations}$\!$}

\begin{document}
\title{On a class fractional Schr\"{o}dinger
equations with indefinite potential involving critical exponential growth
 \thanks{Research partially supported by
CNPq and CAPES.}}
\author{Manass\'{e}s de Souza\footnote{Corresponding author, {\small  E-mail address:
manassesxavier@hotmail.com}}  and Yane Lisley Araújo
\\{\small Departamento de Matem\'{a}tica}
\\{\small Universidade Federal da Para\'{\i}ba}\\
{\small  58051-900 Jo\~{a}o Pessoa, PB, Brazil}}
\maketitle
\begin{abstract}
\noindent It is established the existence and multiplicity of weak solutions for
a class of nonlocal equations involving the fractional laplacian,
nonlinearities with critical exponential growth and potentials
this is which may change sign. The proofs of our existence results
rely on minimization methods in combination with the mountain-pass theorem.
\end{abstract}
\section{Introduction}

\hspace{0,6cm}In this paper we investigate the existence and
multiplicity of weak solutions for the following class of
equations
\begin{equation}\label{problema}
(-\Delta)^{1/2}u+V(x)u=f(x,u)+h\quad \text{in} \quad \R,
\end{equation}
where $V: \mathbb{R}
\rightarrow \mathbb{R}$ is a continuous potential which may change
sign, the nonlinearity $f(x,s)$ behaves like $\exp(\alpha_0 s^2)$
when $|s|\rightarrow+\infty$ for some $\alpha_0 >0$, $h$
belongs to the dual of an appropriated functional space and $(-\Delta)^{1/2}$ is the fractional
laplacian. The fractional laplacian $(-\Delta)^{1/2}$ of a mensurable function $u:\R\rightarrow \R$ is defined by
\begin{equation}\label{l}
(-\Delta)^{1/2}u(x)=-\dfrac{1}{2
\pi}\int_{\R}\dfrac{u(x+y)+u(x-y)-2u(x)}{|y|^{2}} \, \ud y.
\end{equation}

Recently, a great attention has been focused on the study of
fractional and non-local operators of elliptic type, both for the
pure mathematical research and in view of concrete real-world
applications. This type of operators arises in a quite natural way
in many different contexts, such as, among the others, the thin
obstacle problem, optimization, finance, phase transitions,
stratified materials, anomalous diffusion, crystal dislocation,
soft thin films, semipermeable membranes, flame propagation,
conservation laws, ultra-relativistic limits of quantum mechanics,
quasi-geostrophic flows, multiple scattering, minimal surfaces,
materials science and water waves. The literature about non-local
operators and on their applications is, therefore, very
interesting and, up to now, quite large (see, for instance,
\cite{DiNezza} for an elementary introduction to this topic and
for a still not exhaustive list of related references).

In order to study variationally \eqref{problema} we will consider a suitable subspace of
the fractional Sobolev space $H^{1/2}(\R)$. We recall that  $H^{1/2}(\R)$ is defined as
\begin{equation*}
H^{1/2}(\mathbb{R}):=\left\{u\in
L^{2}(\mathbb{R})\,:\;\dfrac{|u(x)-u(y)|}{|x-y|}\in
L^{2}(\mathbb{R}\times \mathbb{R})\right\};
\end{equation*}
endowed with the norm
\begin{equation*}
\|u\|_{{1/2,2}}:=\left(\int_{\mathbb{R}^{2}}\dfrac{|u(x)-u(y)|^{2}}{|x-y|^{2}}\,
\ud x \, \ud y+ \int_{\mathbb{R}}|u|^{2}\, \ud
x\right)^{{1}/{2}}.
\end{equation*}
The term
\begin{equation*}
[u]_{{1/2,2}}:=\left(\int_{\mathbb{R}^{2}}\dfrac{|u(x)-u(y)|^{2}}{|x-y|^{2}}\,
\ud x \,\ud y\right)^{{1}/{2}}
\end{equation*}
is the so-called \textit{Gagliardo semi-norm} of function $u$.

We assume suitable conditions on the potential $V(x)$ with which
we will be able to consider a variational framework based in the
space $X$ given by
\[
X=\left\{u \in H^{1/2}(\R)\, : \, \int_{\R}V(x)u^{2} \, \ud x < \infty \right\}.
\]
More precisely, we
assume throughout this paper the following assumptions on $V(x)$:

\begin{enumerate}
\item[$(V_1)$] There exists a positive constant $B$ such that
\begin{equation*}
V(x) \geq -B \ \text{for all} \ x\in \R;
\end{equation*}
\item[$(V_2)$] The infimum
\[
\lambda_{1}:=\displaystyle{\inf_{u \in X
\atop{\|u\|_{2}=1}}\left(\dfrac{1}{2\pi}
\int_{\R^{2}}\dfrac{(u(x)-u(y))^{2}}{|x-y|^{2}}\, \ud x\, \ud y +
\int_{\R} V(x)u^{2}\, \ud x\right)}
\]
is positive;
\item[$(V_{3})$]$\displaystyle{\lim_{R\rightarrow \infty} \nu(\R \setminus \overline{B}_{R})} = +\infty$, where
\[
\nu(G) = \left\{%
\begin{array}{ll}
   \displaystyle{\inf_{u\in X_0(G)\atop{\|u\|_2=1}}}\frac{1}{2\pi}\displaystyle\int_{\R^2}
\dfrac{(u(x)-u(y))^{2}}{|x-y|^{2}} \, \ud x\, \ud y +
\displaystyle\int_{G}V(x)u^{2}\, \ud x & \hbox{if}\quad G\neq\varnothing;\\
    \infty & \hbox{if} \quad G = \varnothing.\\
\end{array}%
\right.
\]
Here $G$ is a open set in $\R$, $X_0(G)= \left\{u \in X\,:\,  u =
0 \quad \text{in} \quad \mathbb{R} \setminus G \right\}$ and
$\overline{B}_{R}$ is the closed ball with center at origin and
radius $R$.
\end{enumerate}

The hypotheses $(V_1)-(V_2)$ will ensure that the space $X$ is
Hilbert when endowed with the inner product
\[
\langle u,v
\rangle=\dfrac{1}{2\pi}\left(\int_{\R^{2}}\dfrac{(u(x)-u(y))(v(x)-v(y))}{|x-y|^{2}}\,
\ud x \, \ud y \right)+ \int_{\R}V(x)uv \, \ud x,\quad u, v
\in X,
\]
to which corresponds the norm $\|u\| = \langle u,u\rangle^{1/2}$
(cf. Section 2).

In this context, we assume that $h\in X^{*}$ (dual space of $X$)
and say that $u\in X$ is a weak solution for the problem
\eqref{problema} if the following equality holds:
\begin{equation}
\dfrac{1}{2\pi}\int_{\R^{2}}\dfrac{(u(x)-u(y))(v(x)-v(y))}{|x-y|^{2}}
\, \ud x \, \ud y + \int_{\R} V(x)uv \, \ud x =\int_{\R}
f(x,u)v \, \ud x+ (h,v),
\end{equation}
for all $v\in X$, where $(\cdot,\cdot)$ denotes the duality
pairing between $X$ and $X^{*}$.

When a weak solution $u$ has sufficient regularity, it is possible to have a
pointwise expression of the fractional laplacian as \eqref{l}. See
\cite{M. Weinstein}, for example.

As we already mentioned, we are interested in the case that the nonlinearity
$f(x,s)$ has the maximal growth which allows us to treat problem
\eqref{problema} variationally in $X$. Precisely, we will assume
sufficient conditions so that a weak solution of \eqref{problema} turn
out to be critical points of the Euler functional $I: X
\rightarrow \R$ defined by
\[
I(u) = \frac{1}{2}\|u\|^2 -\int_{\R} F(x,u)\, \ud x - (h,u),
\quad\mbox{where}\quad F(x,s)=\int_0^s f(x,t)\ud t\,.
\]

In order to better describe the hypotheses on $f(x,s)$ we recall
some well known facts about the limiting Sobolev embedding theorem
in $1$-dimension. If  $s \in (0,1/2)$, Sobolev embedding states
that $H^{s}(\mathbb{R}) \hookrightarrow L^{2_s^*}(\mathbb{R})$,
where $2_s^* = 2/(1-2s)$, for this case, the maximal growth of the
nonlinearity $f(x,s)$ which allows to treat problem \eqref{problema}
variationally in $H^{s}(\mathbb{R})$ is given by $|s|^{2_s^*}$ as
$|s| \to +\infty$. If $s=1/2$, Sobolev embedding states that
$H^{1/2}(\mathbb{R}) \hookrightarrow L^{q}(\mathbb{R})$ for any
$q\in[2,+\infty)$, but $H^{1/2}(\mathbb{R})$ is not continuous
embedded in $L^{\infty}(\mathbb{R})$, for details see
\cite{DiNezza, Ozawa}. In the borderline case $s=1/2$, the maximal
growth which allows to treat problem \eqref{problema} variationally in
$H^{1/2}(\mathbb{R})$ is motivated by Trudinger-Moser inequality
proved by  H. Kozono, T. Sato and H. Wadade \cite{Wadade} and T.
Ozawa \cite{Ozawa}. Precisely, they proved that there exist
positive constants $\omega$ and $C$ such that for all $u\in
H^{1/2}(\R)$ with $\|(-\Delta)^{1/4}u\|_{2} \leq 1$,
\begin{equation}\label{Teorema1}
\int_\R(e^{\alpha u^{2}}-1) \, \ud x \leq C \|u\|_2^2\,,\quad
\forall\,\, \alpha \in (0, \omega].
\end{equation}
See also the pioneering works \cite{Moser71,Trudinger}. Motivated
by \eqref{Teorema1} we say that $f(x,s)$ has \textit{critical
exponential growth} when for all $x\in \R$, there exists
$\alpha_{0}>0$ such that
\begin{equation*}
\lim_{|s|\rightarrow
+\infty}f(x,s)e^{-\alpha|s|^{2}}=\begin{cases} 0, &  \forall
\alpha > \alpha_{0}, \\ +\infty, &  \forall \alpha <
\alpha_{0}.\end{cases}
\end{equation*}

Now we state our main assumptions for the nonlinearity $f(x,s)$.
In order to find weak solutions \eqref{problema} through
variational methods we will assume the following general
hypotheses:

\begin{enumerate}
\item[$(f_1)$] $0 \leq \displaystyle{\lim_{s\rightarrow 0} \dfrac{f(x,s)}{s}}< \lambda_1$, uniformly in $x$;
\item[$(f_2)$] $f:\R\times \mathbb{R}\rightarrow \mathbb{R}$ is continuous, it
has critical exponential growth and is locally bounded in $s$,
that is, for any bounded interval $J\subset \R$, there exists
$C>0$ such that $ |f(x,s)|\leq C$ for every $(x,s)\in \R\times
J$;
\item[$(f_3)$] there exists $\theta>2$ such that
\begin{equation*}
0<\theta F(x,s):=\theta \int_{0}^{s} f(x,t)\, \ud t \leq s f(x,s),
\quad\mbox{for all}\quad (x,s)\in \R\times\R\setminus\{0\};
\end{equation*}
\item[$(f_4)$]there exist constants $s_0,M_{0} > 0$ such that
\begin{equation*}
0<F(x,s)\leq M_{0}|f(x,s)|, \quad\mbox{for all}\quad  |s|\geq s_0
\quad\mbox{and}\quad x\in \R;
\end{equation*}
\item[$(f_5)$] there exist constants $p > 2$ and $C_p$ such that, for all
$s\geq0$ and $x \in \R$,
\begin{equation*}
f(x,s)\geq C_{p}s^{p-1},
\end{equation*}
with $C_{p}> \left[\dfrac{\alpha_{0}(p-2)}{2\pi\kappa\omega p
}\right]^{(p-2)/2}S_{p}^{p}$, where $S_{p}$ is given in
\eqref{SQ}.
\end{enumerate}

We point out that the hypotheses $(f_1)-(f_5)$ have been used in
many papers to find a variational solution using the classical
Mountain Pass Theorem introduced by Ambrosetti and Rabinowitz in
the celebrated paper \cite{A. Ambrosetti and P. H. Rabinowitz},
see for instance \cite{Manasses,doO,Iannizzotto}. A simple model
of a function that verifies our assumptions is $f(x,s) =  C_p
|s|^{p-2}s+2s(e^{s^2}-1)$ for $(x,s) \in \R \times \R$.

\medskip

We next state our main results.

\medskip

\begin{thm}\label{teodeexistencia}
Suppose that $(V_{1})-(V_{3})$ and $(f_1)-(f_5)$ are satisfied.
Then there exists $\delta_{1}>0$ such that for
each $0<\|h\|_{*}<\delta_{1}$, problem
$\eqref{problema}$ has at least two weak solutions. One of them
with positive energy, while the other one with negative energy.
\end{thm}

\bigskip

\begin{thm}\label{teodeexistencia2}
Under the same hypotheses in Theorem \ref{teodeexistencia}, the
problem without the perturbation, that is $h\equiv 0$, has a
nontrivial weak solution with positive energy.
\end{thm}

\bigskip

\begin{rem}
Our work was motivated by Iannizzotto and Squassina
\cite{Iannizzotto} and some papers that have appeared in the
recent years concerning the study of \eqref{problema}  by using
purely variational approach, see
\cite{Ming,Servadei,Simone} and references therein.
Our goal is to extend and complement the results in
\cite{Ming,Iannizzotto,Servadei,Simone} in sense that we consider
critical exponential growth on the nonlinearity and a class of
potentials $V(x)$ which may change sign, vanish and be
unbounded.
\end{rem}

\begin{rem}
By examining the literature, we notice that many authors,
considering different ways, have established the existence of
solutions for problems involving the standard laplacian
\begin{equation}\label{ppppp}
 -\Delta u +V(x)u=g(x,u),\quad x \in \mathbb{R}^N,
\end{equation}
see e.g. \cite{A. Ambrosetti and P. H. Rabinowitz,BW} for the case
where $g(x,s)$ has subcritical growth in the Sobolev sense, and
see e.g. \cite{Manasses,doO,LamLu,Yang} for the case where
$g(x,s)$ has critical growth in the Trudinger-Moser sense. The
existence of solutions has been discussed under various conditions
on the potential $V(x)$. It is worthwhile to remark that in these
works different hypotheses are assumed on $V(x)$ in order to
overcome the problem of ``lack of compactness'', typical of
elliptic problems defined in unbounded domains and involving
nonlinearities in critical growth range. Specifically, in
\cite{BW,Rabinowitz2} it is assumed that the potential is continuous
and uniformly positive. Furthermore, it is assumed  one of the
following conditions:
\begin{description}
  \item[(a)] $V(x) \nearrow +\infty$ as $|x|\rightarrow+\infty$;
  \item[(b)] For any  $A>0$,  the level set $\{x\in
    \mathbb{R}^N \,:\,V(x)\leq A\}$ has finite Lebesgue measure.
\end{description}

Each of these conditions guarantee that the space
\[ E:=\left\{u\in
W^{1,2}(\mathbb{R}^N): \int_{\mathbb{R}^N}V(x)u^2\,\ud x<\infty
\right\} \] is compactly embedded in the Lebesgue space
$L^q(\mathbb{R}^N)$ for all $q\geq N$.

We point out $ (V_3) $ generalizes these two conditions. In
special, it should be mentioned that the conditions $(V_1)-(V_3)$ were
already considered by B.~Sirakov \cite{Siracov} to study
\eqref{ppppp} when $g(x,u)$ has subcritical growth in the Sobolev
sense.
\end{rem}

\begin{rem}
Similar to \cite{Miyagaki,Manasses,doO,Iannizzotto} we will use
minimization to find the first solution with negative energy, and
the Mountain Pass Theorem to guarantee the existence of the second
solution with positive energy. First we need to check some
conditions concerning the mountain pass geometry and the
compactness of the associated-Euler functional. In our argument,
it is crucial a version of the Trudinger-Moser inequality to space
the $X$ and a version of a Concentration-Compactness Principle due
to P. -L. Lions \cite{PLLions} to the space $X$ (cf. Section
\ref{section1}). Our main difficulties are the involved operator
which is nonlocal and critical exponential growth on the
nonlinearity.
\end{rem}

\medskip

\textit{The outline of the paper is as follows:} Section 2
contains some preliminary results. Section 3 contains the
variational framework and we also check the geometric conditions
of the associated functional. Section 4 deals with Palais-Smale
condition and Section 5 treats with the minimax level. Finally in
Section 6, we complete the proofs of our main results.

Hereafter, $C$, $C_0$, $C_{1}$, $C_{2}$, ... denote positive
constants (possibly different), we use the notation $\|\cdot\|_p$
for the standard $L^p(\mathbb{R})$-norm and $\|\cdot\|_{*}$ for
the norm in the dual space $X^{*}$.

\medskip

\section{Some Preliminary Results}\label{section1}

In this section, we prove some technical results about the space
$X$ and we show a version of \eqref{Teorema1} to $X$. First, in
order to obtain good properties for $ X $, we need the following
lemma:

\begin{lemma}\label{lema1}
Suppose that $(V_{1})$ and $(V_{2})$ are satisfied. Then there
exists $\kappa>0$ such that for any $u \in X$,
\begin{equation}\label{equacao1}
\dfrac{1}{2\pi}\left(\int_{\R^{2}}\dfrac{(u(x)-u(y))^{2}}{|x-y|^{2}}\,\ud x\, \ud y\right)+ \int_{\R}V(x)u^{2}\,\ud x \geq \kappa\|u\|_{1/2,2}^{2}.
\end{equation}
\end{lemma}

\begin{proof} Suppose that \eqref{equacao1} is false. Then for each $n \in \mathbb{N}$ there exists $u_{n} \in X$ such that $\|u_{n}\|^{2}_{1/2,2}=1$ and
\[
\dfrac{1}{2\pi}\left(\int_{\R^{2}}\dfrac{(u_{n}(x)-u_{n}(y))^{2}}{|x-y|^{2}}\,\ud x \,\ud y\right)+ \int_{\R}V(x)u_{n}^{2}\,\ud x < \dfrac{1}{n}.
\]
Thus, by $(V_2)$ it follows that
\begin{eqnarray*}
\lambda_{1} & \leq &
\dfrac{1}{\|u_{n}\|^{2}_{2}}\left(\dfrac{1}{2\pi}\left(\int_{\R^{2}}\dfrac{(u_{n}(x)-u_{n}(y))^{2}}{|x-y|^{2}}
\,\ud x\, \ud y\right)+\int_{\R}V(x)u_{n}^{2}\,\ud x\right) <
\dfrac{1}{n\|u_{n}\|^{2}_{2}}.
\end{eqnarray*}
This together with $\lambda _{1}>0$ imply $\|u_{n}\|_{2}
\rightarrow 0 $ and $[u_n]_{1/2,2} \rightarrow 1$. Consequently, by using $(V_1)$
we obtain the contradiction
\begin{eqnarray*}
o_{n}(1)=-B\|u_{n}\|^{2}_{2} \leq  \int_{\R}V(x)\,u_{n}^{2}\,\ud x
 <  \dfrac{1}{n}-\dfrac{1}{2\pi}[u_n]_{{1/2,2}} \rightarrow
 -\dfrac{1}{2\pi}.
\end{eqnarray*}
This completes the proof.
\end{proof}

Using \eqref{equacao1} we can define the following inner product
in $X$,
\begin{equation}\label{prodinterno}
\langle u,v
\rangle:=\dfrac{1}{2\pi}\left(\int_{\R^{2}}\dfrac{(u(x)-u(y))(v(x)-v(y))}{|x-y|^{2}}\,
\ud x \, \ud y \right)+ \int_{\R}V(x)uv \, \ud x,\quad u, v \in X,
\end{equation}
to which corresponds the norm
\[
\|u\| =
\left\{\dfrac{1}{2\pi}\left(\int_{\R^{2}}\dfrac{|u(x)-u(y)|^2}{|x-y|^{2}}\,
\ud x \, \ud y \right)+ \int_{\R}V(x)u^2 \, \ud x
\right\}^{1/2},\quad u \in X.
\]
Moreover, the following facts hold:
\begin{description}
  \item[\emph{i)}] $X$ is a Hilbert space;
  \item[\emph{ii)}] $X$ is continuously embedded into $H^{1/2}(\R)$;
  \item[\emph{iii)}] For any $q \in
    [2,\infty)$, $X$ is continuously embedded into $L^{q}(\R)$  and
    \begin{equation}\label{SQ} S_{p}:=\displaystyle{\inf_{u
\in X \atop{\|u\|_{p}=1}}\left(\dfrac{1}{2\pi}
\int_{\R^{2}}\dfrac{(u(x)-u(y))^{2}}{|x-y|^{2}}\, \ud x\, \ud y +
\int_{\R} V(x)u^{2}\, \ud x\right)^{1/2}} >0.
\end{equation}
\end{description}

Now, by adapting the arguments in \cite{Siracov}, we prove a
compactness result which will be used in this paper.

\begin{lemma}\label{lema3.1}
Suppose $(V_1)-(V_3)$ hold. Then $X$ is compactly embedded into
$L^{q}(\R)$ for any $q \in [2,\infty)$.
\end{lemma}

\begin{proof}
Let $(u_{n})\subset X$ be a bounded sequence, up to a subsequence,
we may assume that $u_{n}\rightharpoonup 0$ weakly in $X$. We must
prove that, up to a subsequence,
\begin{equation*}
u_{n} \rightarrow 0 \quad \text{strongly in}\quad L^{2}(\R).
\end{equation*}
Let $\varphi\in C^{\infty}(\R,[0,1])$ such that $\varphi \equiv 0$
in $\overline{B}_R$ and $\varphi \equiv 1$ in $\R \backslash
\overline{B}_{R+1}$. Then,
\begin{equation}\label{7est}
\begin{aligned}
\|u_{n}\|_{2}& =  \|(1-\varphi)u_{n}+\varphi u_{n}\|_{2}\\
 & \leq  \|(1-\varphi)u_{n}\|_{2}+ \|\varphi
u_{n}\|_{2}\\  & =  \|(1-\varphi)u_{n}\|_{L^{2}(B_{R+1})} +
\|\varphi u_{n}\|_{L^{2}(\R\setminus \overline{B}_{R})}.
\end{aligned}
\end{equation}
Since $H^{1/2}(B_{R+1})$ is compactly embedded into
$L^{2}(B_{R+1})$, up to a subsequence,
\begin{equation}\label{8est'}
\|(1-\varphi)u_{n}\|_{L^{2}(B_{R+1})}\rightarrow 0.
\end{equation}
Now, by the definition of $\nu(\R\setminus \overline{B}_{R})$, it
follows that
\begin{equation*}\label{9est'}
\|\varphi u_{n}\|^2_{L^{2}(\R\setminus \overline{B}_{R})}\leq
\dfrac{\|\varphi u_{n}\|^{2}}{\nu(\R\setminus \overline{B}_{R})}
\leq\dfrac{C}{\nu(\R\setminus \overline{B}_{R})},
\end{equation*}
which together with $(V_{3})$, implies
\begin{equation}\label{10est'}
\|\varphi u_{n}\|_{L^{2}(\R\setminus \overline{B}_{R})}\rightarrow
0.
\end{equation}
Finally, from \eqref{7est}, \eqref{8est'} and \eqref{10est'}, we
conclude the proof.
\end{proof}

In the sequel we shall prove a version of \eqref{Teorema1} to the
space $X$, this will be our main tool to prove Theorems
\ref{teodeexistencia} and \ref{teodeexistencia2}. The ideas used
in the proof are inspired in \cite{Manasses, doO,Iannizzotto} and
we present here for completeness of our work. For this, we need
the following relation
\begin{equation}\label{sfr1}
\|(-\Delta)^{1/4}u\|_{2}=(2\pi)^{-1/2}[u]_{1/2,2},\quad \forall
\,u \in H^{1/2}(\R),
\end{equation}
for details see \cite[Proposition~3.6]{DiNezza}.

\begin{lemma}\label{lem}
If $0< \alpha  \leq 2\pi\kappa\omega$ and $u \in X$ with $\|u\| \leq 1$,
then there exists a constant $C>0$ such that
\begin{equation}\label{Corolario1}
\int_\R(e^{\alpha u^{2}}-1) \, \ud x \leq C.
\end{equation}
Moreover, for any $\alpha >0$ and $u \in X$, we have
\begin{equation}\label{lema2.3}
\int_\R(e^{\alpha u^{2}}-1) \, \ud x < \infty.
\end{equation}
\end{lemma}

\begin{proof}
First, we observe that if a function $u\in X$ satisfies $\|u\| \leq
1$, set $v=(2\pi\kappa)^{1/2}u$, then $v\in H^{1/2}(\R)$ and by
using \eqref{equacao1} and \eqref{sfr1}, we get
\begin{equation*}
\|(-\Delta)^{1/4}v\|_{2}=(2\pi)^{-1/2}[v]_{1/2,2} \leq
\kappa^{1/2}\|u\|_{1/2,2}\leq \|u\| \leq 1.
\end{equation*}
Consequently,
\begin{eqnarray*}
\int_{\R}(e^{\alpha u^{2}}-1)\, \ud x =\int_{\R}(e^{(\alpha/2\pi
\kappa)v^{2}}-1)\, \ud x \leq C_1\|v\|_{2}^{2} \leq C,
\end{eqnarray*}
where we have used \eqref{Teorema1}. Thus, we obtain
\eqref{Corolario1}.

Now we prove the second part of the lemma. Indeed,  given $u \in
X$ and $\varepsilon>0$ there exists $\varphi \in
C_{0}^{\infty}(\R)$ such that $\|u-\varphi\|<\varepsilon$. Thus,
since
\begin{equation*}
e^{\alpha u^{2}}-1 \leq
e^{\alpha(2(u-\varphi)^{2}+2\varphi^{2})}-1 \leq
\dfrac{1}{2}\left(e^{4\alpha
(u-\varphi)^{2}}-1\right)+\dfrac{1}{2}\left(e^{4\alpha
\varphi^{2}}-1\right),
\end{equation*}
it follows that
\begin{equation}\label{jjjj1}
\int_{\R}(e^{\alpha u^{2}}-1) \, \ud x \leq
\dfrac{1}{2}\int_{\R}(e^{4\alpha\|u-\varphi\|^{2}\left(\frac{u-\varphi}{\|u-\varphi\|}\right)^{2}}
-1) \, \ud x + \dfrac{1}{2}\int_{\R}(e^{4\alpha\varphi^{2}}-1) \,
\ud x.
\end{equation}
Choosing $\varepsilon >0$ such that ${4\alpha\varepsilon^{2}} <
2\pi\kappa\omega$, we have ${4\alpha\|u-\varphi\|^{2}}< 2\pi\kappa\omega$. Then, from
\eqref{Corolario1} and \eqref{jjjj1}, we obtain
that
\begin{equation*}
\int_{\R}(e^{\alpha u^{2}}-1) \, \ud x \leq \dfrac{C}{2} +
\dfrac{1}{2} \int_{\supp(\varphi)}(e^{4\alpha \varphi^{2}}-1) \,
\ud x< \infty.
\end{equation*}
This completes the proof.
\end{proof}

The next result will be used to ensure the geometry of the
functional $I$ in Section \ref{section2}.

\begin{lemma}\label{lema3.2}
If $v \in X$, $\alpha >0$, $q>2$ and $\|v\|\leq M$ with $\alpha
M^{2}<2\pi\kappa\omega$, then there exists $C=C(\alpha,M,q)>0$
such that
\begin{equation*}
\int_\R(e^{\alpha v^{2}}-1)|v|^{q}\ud x \leq C\|v\|^{q}.
\end{equation*}
\end{lemma}

\begin{proof}
We consider $r > 1$ close to $1$ such that $\alpha r M^{2}<2\pi \kappa \omega$ and $r'q \geq 2$, where $r'={r}/{(r-1)}$.\\
Using H\"{o}lder inequality, we have
\begin{equation}\label{exponential11}
\int_\R(e^{\alpha v^{2}}-1)|v|^{q}\ud x \leq
\left(\int_{\R}(e^{\alpha v^{2}}-1)^{r}\ud
x\right)^{1/r}\|v\|_{r'q}^{q}.
\end{equation}
Note that given $\beta >r$ there exists $C=C(\beta) >0$ such that
for all $s \in \mathbb{R}$,
\begin{equation}\label{exponential}
(e^{\alpha s^{2}}-1)^{r} \leq C(e^{\alpha \beta s^{2}}-1).
\end{equation}
Hence, from \eqref{exponential11} and \eqref{exponential}, we get
\begin{eqnarray*}
\int_{\R}(e^{\alpha v^{2}}-1)|v|^{q}\ud x &\leq & C \left(
\int_\R(e^{\alpha \beta v^{2}}-1)\ud
x\right)^{1/r}\|v\|_{r'q}^{q}\\
& \leq & C \left(\int_{\R}(e^{\alpha\beta M^{2}
\left(\frac{v}{\|v\|}\right)^{2}}-1)\,\ud x \right)^{1/r}
\|v\|_{r'q}^q.
\end{eqnarray*}
By choosing $\beta >r$ close to $r$, in such way that $\alpha
\beta M^{2}<2\pi\kappa\omega$, it follows from
\eqref{Corolario1} and the continuous embedding $X \hookrightarrow L^{r'q}(\R)$ that
\begin{equation*}
\int_{\R}(e^{\alpha v^{2}}-1)|v|^{q}\,\ud x \leq C\|v\|^q.
\end{equation*}
This completes the proof.
\end{proof}

\bigskip

Now, in line with the Concentration-Compactness Principle due to
P. -L. Lions \cite{PLLions} we show a refinement of
\eqref{Corolario1}. This result will be crucial to show that the
functional $I$ satisfies the Palais-Smale condition.

\begin{lemma}\label{tipolions}
If $(v_{n})$ is a sequence in $X$ with $\|v_{n}\|=1$ for all $n
\in \mathbb{N}$ and $v_{n}\rightharpoonup v$ in $X$, $0<\|v\|<1$,
then for all $0<t<2\pi\kappa\omega(1-\|v\|^{2})^{-1}$, we have
\begin{equation*}
\sup_{n} \int_{\R}(e^{tv_{n}^{2}}-1)\,\ud x<\infty.
\end{equation*}
\end{lemma}
\begin{proof}
Since $v_{n}\rightharpoonup v$ in $X$ and $\|v_{n}\|=1$, we
conclude that
\begin{equation*}
\|v_{n}-v\|^{2}=1-2\langle v_{n},v \rangle+\|v\|^{2}\rightarrow
1-\|v\|^{2}<\dfrac{2\pi\kappa\omega}{t}.
\end{equation*}
Thus, for $n \in \mathbb{N}$ enough large, we have
$t\|v_{n}-v\|^{2}<2\pi\kappa\omega$. Now choosing $q>1$ close to 1
and $\varepsilon>0$ satisfying
\begin{equation*}
qt(1+\varepsilon^{2})\|v_{n}-v\|^{2}<2\pi\kappa\omega.
\end{equation*}
Consequently, by \eqref{Corolario1}, there exists $C>0$ such that
\begin{equation}\label{26}
\int_{\R}(e^{qt(1+\varepsilon^{2})(v_{n}-v)^{2}}-1) \, \ud x=
\int_{\R}\left(e^{qt(1+\varepsilon)^{2}\|v_{n}-v\|^{2}\left(\frac{v_{n}-v}{\|v_{n}-v\|}\right)^{2}}-1\right)
\, \ud x \leq C.
\end{equation}
Moreover, since
\begin{equation*}
tv_{n}^{2}\leq
t(1+\varepsilon^{2})(v_{n}-v)^{2}+t\left(1+\dfrac{1}{\varepsilon^{2}}\right)v^{2},
\end{equation*}
it follows by convexity of the exponential function with
$q^{-1}+r^{-1}=1$ that
\begin{eqnarray*}
e^{tv_{n}^{2}}-1\leq
\dfrac{1}{q}(e^{qt(1+\varepsilon^{2})(v_{n}-v)^{2}}-1)+\dfrac{1}{r}(e^{rt(1+1/\varepsilon^{2})v^{2}}-1).
\end{eqnarray*}
Therefore, by \eqref{lema2.3} and \eqref{26}, we get
\begin{equation*}
\int_{\R}(e^{tv_{n}^{2}}-1)\, \ud x \leq
\dfrac{1}{q}\int_{\R}(e^{qt(1+\varepsilon^{2})(v_{n}-v)^{2}}-1) \,
\ud x + \int_{\R}(e^{rt(1+1/\varepsilon^{2})v^{2}}-1)\, \ud x \leq
C,
\end{equation*}
and the result is proved.
\end{proof}

\section{The variational framework}\label{section2}

\hspace{0,6 cm}As we mentioned in the introduction, the problem
\eqref{problema} has variational structure. In order to apply the
critical point theory, we define the following functional $I: X
\rightarrow \mathbb{R}$,
\[
I(u)=\dfrac{1}{2}\|u\|^{2}-\int_{\R}F(x,u)\, \ud x- (h,u).
\]
Notice that, from $(f_1)$ and $(f_2)$, for each $\alpha>\alpha_0$
and $\varepsilon>0$ there exists $C_\varepsilon >0$ such that
\[
|F(x,s)| \leq \dfrac{(\lambda_1-\varepsilon)}{2}  s^2 +
C_\varepsilon (e^{\alpha s^2}-1),\label{V}\quad\forall \,s\in
\mathbb{R},
\]
which together with the continuous embedding $X \hookrightarrow
L^2(\R)$ and \eqref{lema2.3} yields $F\left(x,u\right)\in L^1(\R)$
for all $u \in X$. Consequently, $I$ is well-defined and by
standard arguments, $I\in C^1(X,\R)$ with
\begin{equation*}
(I'(u),v)=\dfrac{1}{2\pi}\int_{\R^{2}}\dfrac{(u(x)-u(y))(v(x)-v(y))}{|x-y|^{2}}
\, \ud x \, \ud y + \int_{\R} V(x)uv \, \ud x -\int_{\R} f(x,u)v
\, \ud x- (h,v),
\end{equation*}
for all $u,v\in X$. Hence, a critical point of $I$ is a weak
solution of \eqref{problema} and reciprocally.

The geometric conditions of the mountain-pass theorem for the
functional $I$ are established by next lemmas.

\begin{lemma}\label{geometria1}
Suppose that $(V_1)-(V_{2})$ and $(f_1)-(f_2)$ hold. Then there exists $\delta_{1}>0$ such that for each
$h\in X^{*}$ with $\|h\|_{*}<\delta_{1} $, there exists
$\rho_{h}>0$ such that
\begin{equation*}
I(u) > 0 \quad \text{if}\quad \|u\|=\rho_{h}.
\end{equation*}
\end{lemma}

\begin{proof}
From $(f_1)$ and $(f_2)$, given $\varepsilon>0$, there exists
$C>0$ such that, for all $\alpha >\alpha_{0}$ and $q>2$,
\begin{equation}\label{estimativaF}
|F(x,s)|\leq \dfrac{(\lambda_1 -\varepsilon)}{2} s^{2}+C(e^{\alpha
s^{2}}-1)|s|^{q},\quad \forall s \in \R.
\end{equation}
Using \eqref{estimativaF} and $(V_2)$, we have
\begin{eqnarray*}
I(u) & \geq & \dfrac{1}{2}\|u\|^{2}- \dfrac{(\lambda_1 -\varepsilon)}{2}
\int_{\R}u^{2}\,\ud x-C \int_{\R}(e^{\alpha
u^{2}}-1)|u|^{q} \, \ud x-\|h\|_{*}\|u\|\\&\geq &
\dfrac{1}{2}\|u\|^{2}-\dfrac{(\lambda_1 -\varepsilon)}{2 \lambda_1}\|u\|^{2}- C\int_{\R}(e^{\alpha
u^{2}}-1)|u|^{q} \, \ud x-\|h\|_{*}\|u\|.
\end{eqnarray*}
Then, for $u\in X$ such that $\alpha\|u\|^{2}<2\pi\kappa\omega$,
using Lemma \ref{lema3.2}, we get
\begin{equation*}
I(u) \geq \left(\dfrac{1}{2}-\dfrac{(\lambda_1 -\varepsilon)}{2\lambda_{1}}\right)\|u\|^{2}-
C\|u\|^{q}-\|h\|_{*}\|u\|.
\end{equation*}
Consequently,
\begin{equation*}
I(u) \geq \|u\|\left[\left(\dfrac{1}{2}-\dfrac{(\lambda_1 -\varepsilon)}{2\lambda_{1}}\right)\|u\|-
C\|u\|^{q-1}-\|h\|_{*}\right].
\end{equation*}
Since $\frac{1}{2}-\frac{ (\lambda_1 -\varepsilon)}{2\lambda_{1}}>
0$, we may choose $\rho_{h}>0$ such that
\begin{equation*}
\left(\dfrac{1}{2}-\dfrac{(\lambda_1 -\varepsilon)}{2\lambda_{1}}\right)\rho_{h}-
C\rho_{h}^{q-1}>0.
\end{equation*}
Thus, for $\|h\|_{*}$ sufficiently small there exists $\rho_{h}$
such that $I(u)>0$ if $\|u\|=\rho_{h}.$
\end{proof}

\begin{lemma}\label{geometria2}
Assume that $(V_{1})-(V_{2})$ and $(f_2)-(f_3)$ hold. Then there exists $e \in X$ with
$\|e\|>\rho_{h}$ such that
\begin{equation*}
I(e)<\inf_{\|u\|= \rho_{h}} I(u).
\end{equation*}
\end{lemma}

\begin{proof}Let $u \in C_{0}^{\infty}(\R)\setminus\{0\}$, $u \geq 0$ with compact support $K=\supp(u)$. By using $(f_2)$ and $(f_{3})$,
there exist positive constants $C_{1}$ and $C_{2}$ such that
\begin{equation*}
F(x,s)\geq C_{1} s^{\theta}-C_{2},\quad \forall \,(x,s)\in K
\times [0,\infty).
\end{equation*}
Then, for $t>0$, we get
\begin{eqnarray*}
I(tu) & \leq & \dfrac{t^{2}}{2}\|u\|^{2}-C_{1} t^{\theta} \int_{K}
u^{\theta} \, \ud x +  C_{2} \int_{K} \ud x + t|(h,u)|.
\end{eqnarray*}
Since $\theta>2$, we have $I(tu) \rightarrow - \infty$ as $t
\rightarrow \infty$. Setting $e=tu$ with $t$ large enough, the
proof is finished.
\end{proof}

\bigskip

In order to find an appropriate ball to use minimization argument,
we need the following results:
\begin{lemma}\label{solnegativa1}
Suppose that $(V_1)-(V_2)$ and $(f_1)-(f_2)$ hold. Then if $h \not = 0$ there exist $\eta>0$ and $v \in X\setminus \{0\}$ such that $I(tv)<0$ for all $0<t<\eta$. In particular,
\begin{equation*}
-\infty<c_{0}\equiv\inf_{\|u\|\leq \eta} I(u)<0.
\end{equation*}
\end{lemma}
\begin{proof}
For each $h \in X^{*}$, by applying the Riesz representation
theorem in the space $X$, the problem
\begin{equation*}
(-\Delta)^{1/2}u+V(x)u=h,\quad x \ \ \text{in} \ \ \R,
\end{equation*}
has a unique weak solution $v \in X$ so that
\begin{equation*}
(h,v)=\|v\|^{2}>0.
\end{equation*}
Consequently, from $(f_1)$ and $(f_2)$ it follows that there exists
$\eta>0$ such that
\begin{equation*}
\dfrac{d}{dt}I(tv)=t\|v\|^{2}-\int_{\R}f(x,tv)v\, \ud x -
(h,v) <0,
\end{equation*}
for all $0<t<\eta$. Using that $I(0)=0$, it must hold that
$I(tv)<0$ for all $0<t<\eta$ and the proof is completed.
\end{proof}

\section{Palais-Smale compactness condition}\label{Palais-Smale
Sequences}

In this section we show that $I$ satisfies the Palais-Smale condition for certain energy levels. We recall that the functional $I$ satisfies the Palais-Smale
condition at level $c$, denoted by $(PS)_{c}$, if for any sequence
$(u_{n})$ in $X$ such that
\begin{equation}\label{seqPS}
I(u_{n})\rightarrow c \quad \text{and} \quad I'(u_{n})\rightarrow
0 \ \ \text{as} \ \ n\rightarrow\infty,
\end{equation}
has a strongly convergent subsequence in $X$.

Initially, we prove the following lemma:

\begin{lemma}\label{lemaconvergencia}
Suppose that $(V_1)-(V_3)$ and $(f_1)-(f_4)$ are satisfied. Let
$(u_{n}) \subset X$ be an arbitrary Palais-Smale sequence of $I$
at level $c$. Then there exists a subsequence of $(u_n)$ (still denoted by
$(u_n)$) and $u \in X$ such that
\[
\left\{%
\begin{array}{ll}
    u_n \rightharpoonup u & \hbox{weakly in}\,\, X, \\
    f(x,u_{n})\rightarrow f(x,u) & \hbox{in}\,\, L_{loc}^1(\R),\\
    F(x,u_{n})\rightarrow F(x,u) & \hbox{in}\,\, L^{1}(\R).\\
\end{array}%
\right.
\]
\end{lemma}
\begin{proof}
Note that by $(f_3)$,
\begin{eqnarray}\label{A}
I(u_{n})-\dfrac{1}{\theta}(I'(u_{n}),u_n) &=&
\left(\dfrac{1}{2}-\dfrac{1}{\theta}\right)\|u_{n}\|^2+
\int_{\R}\left(\dfrac{1}{\theta}f(x,u_{n})u_{n}-F(x,u_{n})\right)\, \ud x + \left(\frac{1}{\theta}-1\right)(h,u_n) \nonumber\\
&\geq& \left(\dfrac{1}{2}-\dfrac{1}{\theta}\right) \|u_{n}\|^{2} +
\left(\frac{1}{\theta}-1\right)(h,u_n).
\end{eqnarray}
By using \eqref{seqPS}, we obtain
\begin{equation*}
I(u_{n})-\dfrac{1}{\theta}(I'(u_{n}),u_n) \leq C+\|u_{n}\|.
\end{equation*}
This together with \eqref{A} leads to $\|u_{n}\|\leq C$. Hence,
since that $X$ is a Hilbert space, up to subsequence, we can
assume that there exists $u \in X$ such that
\[
\left\{%
\begin{array}{ll}
  u_{n}\rightharpoonup u & \hbox{weakly in}\,\, X, \\
  u_{n}\rightarrow u & \hbox{in}\,\, L^q(\R),\,\, \forall\,\, q \in [2,\infty), \\
   u_{n}(x)\rightarrow u(x)& \hbox{almost everywhere in}\,\, \R. \\
\end{array}%
\right.
\]
Now, from (\ref{seqPS}) and $\|u_{n}\|\leq C$, there exists
$C_1>0$ such that
\begin{equation*}
\int_{\R}|f(x,u_{n})u_{n}|\leq C_1.
\end{equation*}
Consequently, thanks to Lemma 2.1 in \cite{Miyagaki}, we get
\begin{equation}\label{maj}
f(x,u_{n})\rightarrow f(x,u) \,\,\, \hbox{in}\,\,\, L_{loc}^1(\R).
\end{equation}

Next, similar to N. Lam and G. Lu \cite{LamLu}, we shall prove the
last convergence of the lemma. Firstly, note that by using $(f_3)$ and $(f_4)$, for each $R>0$ there exists $C_0
> 0$ such that
\begin{equation*}
F(x,u_{n})\leq C_{0}|f(x,u_{n})|.
\end{equation*}
This together with \eqref{maj} and the Generalized Lebesgue's
Dominated Convergence Theorem, imply
\begin{equation*}
F(x,u_{n})\rightarrow F(x,u) \,\, \text{in} \,\, L^{1}(B_{R}),\,\,
\forall R>0.
\end{equation*}
Now, to conclude the last convergence of the lemma, it is sufficient to prove that given $\delta>0$, there exists $R>0$ such that
\begin{equation*}
\int_{B^{c}_{R}}F(x,u_{n})\,\ud x \leq \delta\,\, \text{and} \,\,
\int_{B^{c}_{R}}F(x,u)\,\ud x \leq \delta.
\end{equation*}
In order to prove it, we notice that by using $(f_{1})$, $(f_{3})$
and $(f_4)$, there exist $C_{1},C_{2}>0$ such that
\begin{equation*}
|F(x,s)|\leq C_{1}|s|^{2}+C_{2} |f(x,s)|, \quad \forall (x,s)\in \R\times\R,
\end{equation*}
Thus, for each $A>0$, we obtain that
\begin{eqnarray*}
\int\limits_{{|x|>R}\atop{|u_{n}|>A}}F(x,u_{n})\, \ud x &\leq &
C_{1}\int\limits_{{|x|>R}\atop{|u_{n}|>A}}|u_{n}|^{2}\, \ud x + C_{2}\int\limits_{{|x|>R}\atop{|u_{n}|>A}}|f(x,u_{n})|\, \ud x \\
&\leq & \dfrac{C_{1}}{A}\int\limits_{{|x|>R}\atop{|u_{n}|>A}}|u_{n}|^{3}\, \ud x+ \dfrac{C_{2}}{A}\int_{\R}|f(x,u_{n})u_{n}|\, \ud x \\
& \leq &
\dfrac{C_{1}}{A}\|u_{n}\|^{3}+\dfrac{C_{2}}{A}\int_{\R}|f(x,u_{n})u_{n}|\,
\ud x.
\end{eqnarray*}
Since $\|u_{n}\|$ and $\int_{\R}|f(x,u_{n})u_{n}|\, \ud x$ are bounded, given $\delta>0$ we may choose $A$ such that
\begin{equation*}
\dfrac{C_{1}}{A}\|u_{n}\|^{3}< \delta/3\quad
\text{and} \quad \dfrac{C_{2}}{A}\int_{\R}|f(x,u_{n})u_{n}|\, \ud x< \delta/3.
\end{equation*}
Thus,
\begin{equation}\label{1}
\int\limits_{{|x|>R}\atop{|u_{n}|>A}}F(x,u_{n})\, \ud x \leq
2\delta/3.
\end{equation}
Moreover, note that with such $A$, by $(f_{1})$ and $(f_{2})$, we
have that
\begin{equation*}
F(x,s)\leq C(\alpha_{0},A)|s|^{2}, \quad \forall(x,s)\in\R \times [-A,A].
\end{equation*}
So, we get
\begin{eqnarray*}
\int\limits_{{|x|>R}\atop{|u_{n}|\leq A}}F(x,u_{n})\, \ud x &\leq
&C(\alpha_{0},A)\int\limits_{{|x|>R}\atop{|u_{n}|\leq A}}
|u_{n}|^{2}\, \ud x \\ &\leq &
2C(\alpha_{0},A)\int\limits_{{|x|>R}\atop{|u_{n}|\leq A}}
|u_{n}-u|^{2}\, \ud x+2C(\alpha_{0},A)\int\limits_{{|x|>R}\atop{|u_{n}|\leq A}}
|u|^{2}\, \ud x .
\end{eqnarray*}
Hence, using Lemma \ref{lema3.1}, given $\delta>0$ we may choose
$R>0$ such that
\begin{equation}\label{2}
\int\limits_{{|x|>R}\atop{|u_{n}|\leq A}}F(x,u_{n})\, \ud x \leq
\delta/3.
\end{equation}
From (\ref{1}) and (\ref{2}), we have that given $\delta>0$ there
exists $R>0$ such that
\begin{equation*}
\int_{|x|>R}F(x,u_{n})\, \ud x \leq \delta.
\end{equation*}
Similarly,
\begin{equation*}
\int_{|x|>R}F(x,u)\, \ud x \leq \delta.
\end{equation*}
Combining all the above estimates and since that $\delta>0$ is arbitrary, we have
\begin{equation*}
\int_{\R}F(x,u_{n}) \, \ud x \rightarrow \int_{\R}F(x,u) \, \ud x,
\end{equation*}
which completes the proof.
\end{proof}

Now, we shall prove main results this Section.

\begin{prop} Under the hypotheses $(V_1)-(V_3)$ and $(f_{1})-(f_{4})$, the
functional $I$ satisfies $(PS)_{c}$ for any
$c<{\pi\kappa\omega}/{\alpha_{0}}$.\label{PS}
\end{prop}
\begin{proof}
Let $(u_{n})\subset X$ be an arbitrary Palais-Smale sequence of $I$ at level $c$. By Lemma \ref{lemaconvergencia}, up to a subsequence, $u_{n}\rightharpoonup u$ weakly in $X$.

We shall show that, up to a subsequence, $u_{n}\rightarrow u$ strongly in $X$. For this, we have two cases to consider:

\bigskip

\noindent \textit{\underline{Case\ 1:}} $u=0$.

\bigskip

\noindent In this case, again by Lemma \ref{lemaconvergencia}, we have
\begin{equation*}
\int_{\R}F(x,u_{n}) \rightarrow 0\quad\mbox{and}\quad (h,u_n)
\rightarrow 0.
\end{equation*}
Since
\begin{equation*}
I(u_{n})=\dfrac{1}{2}\|u_{n}\|^{2}-\int_{\R}F(x,u_{n})-
(h,u_{n})=c+o_{n}(1),
\end{equation*}
we get
\begin{equation*}
\lim_{n\rightarrow\infty}\|u_{n}\|^{2}=2c.
\end{equation*}
Hence, we can infer that for $n$ large there exist $r_{1}>1$
sufficiently close to $1$, $\alpha>\alpha_{0}$ close to
$\alpha_{0}$ and $\tilde{r}_{1}>r_{1}$ sufficiently close to
$r_{1}$ such that $\tilde{r}_{1}\alpha
\|u_{n}\|^{2}<2\pi\kappa\omega$. Thus, by \eqref{exponential} and
\eqref{Corolario1}, we have
\begin{equation}\label{*}
\int_{\R}(e^{\alpha u_{n}^{2}}-1)^{r_{1}}\, \ud x\leq
C\int_{\R}(e^{\tilde{r}_{1}\alpha\|u_{n}\|^{2}\left(\tiny{\frac{u_{n}}{\|u_{n}\|}}\right)^{2}}-1)\,
\ud x \leq C.
\end{equation}
Consequently,
\begin{equation*}
\int_{\R}f(x,u_{n})u_{n}\, \ud x\rightarrow0.
\end{equation*}
In fact, since $f(x,s)$ satisfies $(f_{1})$ and $(f_{2})$, for
$\alpha>\alpha_0$ and $\varepsilon>0$, there exists $C_1 >0$ such
that
\[
|f(x,s)| \leq (\lambda_1-\varepsilon)  |s| + C_1 (e^{\alpha
s^2}-1),\label{V}\quad\forall s\in \mathbb{R}.
\]
Now, letting $r_{1}>1$ sufficiently close to $1$ such that
$r_{2}\geq 2$, where $1/r_{1}+1/r_{2}=1$, we obtain by H\"{o}lder
inequality that
\begin{equation*}
\left|\int_{\R}f(x,u_{n})u_{n}\, \ud x\right| \leq C
\int_{\R}|u_{n}|^{2}\, \ud x + C\left(\int_{\R}(e^{\alpha
u_{n}^{2}}-1)^{r_{1}}\, \ud x \right)^{1/r_{1}}
\left(\int_{\R}|u_{n}|^{r_{2}}\, \ud x\right)^{1/r_{2}}
\rightarrow 0,
\end{equation*}
where we have used $(\ref{*})$ and Lemma \ref{lema3.1}.

Therefore, since $(I'(u_{n}),u_{n})=o_{n}(1)$, we conclude that,
up to a subsequence, $u_{n}\rightarrow0$ strongly in $X$.

\bigskip

\noindent \textit{\underline{Case\ 2:}} $u\neq0$.

\bigskip

\noindent In this case, we define
\begin{equation*}
v_{n}=\dfrac{u_{n}}{\|u_{n}\|} \quad \text{and} \quad
v=\dfrac{u}{\lim\|u_{n}\|}.
\end{equation*}
It follows that $v_{n}\rightharpoonup v$ in $X$, $\|v_{n}\|=1$ and
$\|v\|\leq 1$. Thus, if $\|v\|=1$, we conclude the proof. Now, if
$\|v\|<1$, we claim that there exist $r_{1}>1$ sufficiently close
to $1$, $\alpha>\alpha_{0}$ close to $\alpha_{0}$ and $\beta>0$
such that
\begin{equation}\label{r1}
r_{1}\alpha\|u_{n}\|^{2}\leq
\beta<2\pi\kappa\omega(1-\|v\|^{2})^{-1}
\end{equation}
for $n\in \mathbb{N}$ large.

Indeed, since $I(u_{n})=c+o_{n}(1)$, it follows that
\begin{equation}\label{3}
\dfrac{1}{2}\lim_{n\rightarrow \infty}\|u_{n}\|^{2}=
c+\int_{\R}F(x,u)\, \ud x+ (h,u).
\end{equation}
Setting
\begin{equation*}
A=\left(c+\int_{\R}F(x,u)\, \ud x+ (h,u)\right)(1-\|v\|^{2}),
\end{equation*}
from $(\ref{3})$ and by definition of $v$, we obtain
\begin{equation*}
A=c-I(u),
\end{equation*}
which together with \eqref{3} imply
\begin{equation*}
\dfrac{1}{2}\lim_{n\rightarrow
\infty}\|u_{n}\|^{2}=\dfrac{A}{1-\|v\|^{2}}=\dfrac{c-I(u)}{1-\|v\|^{2}}<\dfrac{\pi\kappa\omega}{\alpha_{0}(1-\|v\|^{2})}.
\end{equation*}
Consequently, \eqref{r1} holds.
Again by \eqref{exponential} and Lemma
\ref{tipolions}, we get
\begin{equation*}
\int_{\R}(e^{\alpha u_{n}^{2}}-1)^{r_{1}}\, \ud x \leq C.
\end{equation*}
By H\"{o}lder inequality and similar computations done above
we have that
\begin{equation*}
\int_{\R}f(x,u_{n})(u_{n}-u)\, \ud x\rightarrow 0.
\end{equation*}
This convergence together with the fact that
$(I'(u_{n}),(u_{n}-u))=o_{n}(1)$ imply that
\begin{equation*}
\|u_{n}\|^{2}=(u_{n},u) +o_{n}(1).
\end{equation*}
Since $u_{n}\rightharpoonup u$ weakly in $X$, we obtain $u_{n}\rightarrow
u$ strongly in $X$ and the proof is finished.
\end{proof}

\section{Minimax Level}

In this section, we verify that the minimax level associated with
the Mountain Pass Theorem is in the interval where the Proposition
\ref{PS} can be applied. To show this result the idea is to find a
nonnegative function $u_{p}\in X$ which attains $S_{p}$. And then
we show the main result of the section, providing the estimate for
$\max_{t\geq 0}I(tu_{p})$.

\begin{lemma}\label{lemasq}
Suppose that $(V_1)-(V_3)$ hold. Then $S_{p}$ is attained by a
non-negative function $u_{p}\in X$.
\end{lemma}
\begin{proof}
Let $(u_{n})$ be a minimizing sequence of non-negative functions
(if necessary, replace $u_{n}$ by $|u_{n}|$) for $S_{p}$ in
$X$, that is,
\begin{equation*}
\|u_{n}\|_{p}=1\quad \text{and} \quad
\left(\dfrac{1}{2\pi} \displaystyle{\int_{\R^{2}}\dfrac{(u_{n}(x)-u_{n}(y))^{2}}{|x-y|^{2}}\, \ud x\, \ud y}
+ \int_{\R} V(x)u_{n}^{2}\, \ud x\right)^{1/2}\rightarrow
S_{p}.
\end{equation*}
Then, $(u_{n})$ is
bounded in $X$. Since $X$ is Hilbert and $X$ is compactly embedded into $L^{p}(\R)$, up to a subsequence, we may assume
\[
\begin{aligned}
& u_n \rightharpoonup u_p \quad\mbox{weakly in}\,\, X,\\
& u_n \rightarrow u_p  \quad\mbox{strongly in}\,\,
L^p(\mathbb{R}),\\
& u_n(x) \rightarrow u_p(x)  \quad\mbox{almost everywhere in}\,\,
\mathbb{R}.
\end{aligned}
\]
Consequently, we have
\begin{equation*}
\|u_{p}\|=1, u_{p}(x)\geq 0 \quad \text{and} \quad \|u_{p}\|\leq \liminf_{n\rightarrow +\infty} \|u_{n}\|=S_{p}.
\end{equation*}
Thus, $S_{p}=\|u_{p}\|$. This completes the proof.
\end{proof}

Now we prove the main result of this section.

\begin{lemma}\label{nivelabaixo}
Suppose that $(V_1)-(V_3)$ and $(f_5)$ are satisfied, if $\|h\|_{*}$ is sufficiently small then
\begin{equation*}
\max_{t\geq 0} I(tu_{p})<\dfrac{\pi\kappa\omega}{\alpha_0}.
\end{equation*}
\end{lemma}
\begin{proof}
Let $\Psi:[0,+\infty)\rightarrow\R$ given by
\begin{equation*}
\Psi(t)=\dfrac{t^{2}}{2}\left(\dfrac{1}{2\pi}\int_{\R^{2}}\dfrac{(u_{p}(x)-u_{p}(y))^{2}}{|x-y|^{2}}\,
\ud x \ud y + \int_{\R} V(x)u_{p}^{2}\,\ud x \right)-\int_{\R}
F(x,tu_{p})\,\ud x.
\end{equation*}
By Lemma \ref{lemasq} and $(f_{5})$, we get
\begin{equation}\label{51}
\Psi(t)\leq\dfrac{t^{2}}{2}S^{2}_{p}-\dfrac{C_{p}}{p}t^{p}\leq
\max_{t\geq 0}\left[\dfrac{t^{2}}{2}S_{p}^{2}-\dfrac{
C_{p}}{p}t^{p}\right]=\dfrac{(p-2)}{2 p}\dfrac{S_{p}^{2p/(p-2)}}{
C_{p}^{2/(p-2)}}<\dfrac{\pi\kappa\omega}{\alpha_{0}}.
\end{equation}
To conclude, note that $|(h,u_{p})|\leq \|h\|_{*}\|u_{p}\|$. Thus taking $\|h\|_{*}$ sufficiently small and using \eqref{51} the result follows.
\end{proof}

\section{Proofs of Theorem \ref{teodeexistencia} and
\ref{teodeexistencia2}}

\hspace{0,6cm}By Lemmas \ref{geometria1}, \ref{geometria2} the
functional $I$ satisfies the geometric properties of the Mountain
Pass Theorem. As a consequence, the minimax level
\begin{equation*}
c_{m}=\inf_{g\in\Gamma} \max_{t\in[0,1]} I(g(t)) >0,
\end{equation*}
where $\Gamma=\{g\in C([0,1],X):g(0)=0,g(1)=e\}$.

By Lemma \ref{nivelabaixo} and Proposition \ref{PS}, the
functional $I$ satisfies the $(PS)_{c_{m}}$ condition. Therefore,
by the mountain-pass Theorem the functional $I$ has a critical
point $u_{m}$ at the minimax level $c_{m}$.

Moreover, if $h \in X^{*}$ with $h \not \equiv 0$, we can find a
second solution. To this, we consider $\rho_{h}$ like in Lemma
\ref{geometria1}. Observe that $\overline{B}_{\rho_{h}}$ is a
complete metric space with the metric induced by the norm of $X$ and
convex, and the functional $I$ is of class $C^{1}$ and bounded
below on $\overline{B}_{\rho_{h}}$. Thus, by Ekeland variational
principle there exists a sequence $(u_{n})$ in
$\overline{B}_{\rho_{h}}$ such that
\begin{equation*}
I(u_{n})\rightarrow c_{0}=\inf_{\|u\|\leq \rho_{h}} I(u)<0 \quad
\text{and} \quad \|I'(u_{n})\|_{*}\rightarrow0.
\end{equation*}
Hence, by Proposition \ref{PS} the functional $I$ satisfies the $(PS)_{c_{0}}$ condition. Consequently, there exists $u_0 \in X$ such that
$I'(u_{0})=0$ and $I(u_0) = c_0$, that is, $u_{0}$ is a weak solution of \eqref{problema} at level $c_{0}$.

Thus it is completed the proof of the results.

\end{document}